\documentclass[10pt,a4paper,reqno]{amsart}
\usepackage[english]{babel}
\usepackage[T1]{fontenc}
\usepackage{indentfirst}
\usepackage{fancyhdr}
\usepackage{amsfonts}
\usepackage{amsmath}
\usepackage{amsthm}
\usepackage{amssymb, enumerate}
\usepackage{hyperref}
\usepackage{theoremref}
\usepackage{bbm}
\newtheorem{thm}{Theorem}[section]
\newtheorem{lem}[thm]{Lemma}
\newtheorem{cor}[thm]{Corollary}
\newtheorem{prop}[thm]{Proposition}

\newtheorem{ques}[thm]{Question}
\theoremstyle{definition}
\newtheorem{rem}[thm]{Remark}

\newtheorem{exmp}[thm]{Example}

\numberwithin{equation}{section}
\newcommand{\be}{\begin{equation}}
\newcommand{\ee}{\end{equation}}
\def\CC{\mathbb{C}}

\def\NN{\mathbb{N}}

\def\RR{\mathbb{R}}

\def\L{{\mathcal L}}

\def\P{{\mathcal P}}

\def\dd{{\mathrm{d}}}
\def\uX{{\underline{X}}}

\newcommand{\Pos}{\mathrm{Pos}}
\newcommand{\supp}{\mathrm{supp}}
\renewcommand{\sp}{\mathfrak{sp}}
\newcommand{\nn}{\mathcal{S}}

\usepackage{tikz}
\usepackage{tikz-cd}

\usepackage[normalem]{ulem}
\definecolor{violet(web)}{rgb}{0.93, 0.51, 0.93}

\makeatletter
\@namedef{subjclassname@2020}{%
  \textup{2020} Mathematics Subject Classification}
\makeatother

\title[A characterization of moment functionals in the compact case]{\bf An intrinsic characterization of\\  moment functionals in the compact case}

\author[Infusino, Kuhlmann, Kuna, Michalski]{Maria Infusino, Salma Kuhlmann, Tobias Kuna, Patrick Michalski}
\address[M. Infusino]{\newline \indent Dipartimento di Matematica e Informatica, Universit\`{a} degli Studi di Cagliari,\newline \indent
Via Ospedale 72,
09124 Cagliari, Italy.}
\email{maria.infusino@unica.it}
\address[S. Kuhlmann, P. Michalski]{\newline \indent Fachbereich Mathematik und Statistik, Universit\"at Konstanz,\newline \indent
Universit\"atstrasse 10,
78457 Konstanz, Germany.}
\email{salma.kuhlmann@uni-konstanz.de}
\email{patrick.michalski@uni-konstanz.de}
\address[T. Kuna]{\newline \indent Dipartimento di Ingegneria, Scienze dell'Informazione e Matematica,\newline \indent Universit\'a degli Studi dell'Aquila, Via Vetoio, 67100, L'Aquila,
Italy}
\email{tobias.kuna@posteo.net}

\subjclass[2020]{Primary 44A60, 28C20, 28C05}
\keywords{moment problem, representing measure, compact support, submultiplicative seminorm, Archimedean quadratic module}
\date{\today}

\begin{document}
\begin{abstract}
We consider the class of all linear functionals~$L$ on a unital commutative real algebra~$A$ that can be represented as an integral w.r.t.\ to a Radon measure with compact support in the character space of~$A$.
Exploiting a recent generalization of the classical Nussbaum theorem, we establish a new characterization of this class of moment functionals solely in terms of a growth condition intrinsic to the given linear functional.
To the best of our knowledge, our result is the first to exactly identify the compact support of the representing Radon measure.
We also describe the compact support in terms of the largest Archimedean quadratic module on which~$L$ is non-negative and in terms of the smallest submultiplicative seminorm w.r.t. which~$L$ is continuous.
Moreover, we derive a formula for computing the measure of each singleton in the compact support, which in turn gives a necessary and sufficient condition for the support to be a finite set.
Finally, some aspects related to our growth condition for topological algebras are also investigated.
\end{abstract}

\maketitle

\section{Introduction}

In this article we investigate the following instance of the moment problem for a unital commutative (not necessarily finitely generated) $\RR$--algebra~$A$. 
We always assume that the \emph{character space} of~$A$, i.e., the set~$X(A)$ of all $\RR$--algebras homomorphisms from~$A$ to~$\RR$, is non-empty and we endow~$X(A)$ with the weakest (Hausdorff) topology~$\tau_{X(A)}$ such that for each~$a\in A$ the function~$\hat{a}\colon X(A)\to\RR$, $\alpha\mapsto\alpha(a)$ is continuous. Our main question is the following.

\begin{ques}
\thlabel{MP}
Let $A$ be a unital commutative $\RR$--algebra with~$X(A)$ non-empty. 
Given a linear functional $L\colon A\to\RR$ with~$L(1)=1$, does there exist a Radon measure~$\nu$ on~$X(A)$ with
\begin{equation}\label{MP::eq1}
L(a)=\int_{X(A)}\hat{a}(\alpha)\dd\nu(\alpha)\qquad\text{for all }a\in A
\end{equation}
such that the support of~$\nu$ is \emph{compact}?
\end{ques}

\noindent
If a Radon measure~$\nu$ as in~\eqref{MP::eq1} does exist, then we call~$\nu$ a \emph{representing Radon measure for~$L$} and we say that~$L$ is a \emph{moment functional}. 
In fact, if the support of a representing Radon measure is compact, then the representation in~\eqref{MP::eq1} is unique (see~\cite[Section~3.3]{Mar08}).
We recall that a \emph{Radon measure}~$\nu$ on~$X(A)$ is a non-negative measure on the Borel~$\sigma$--algebra w.r.t.\ $\tau_{X(A)}$ that is locally finite and inner regular w.r.t.\ compact subsets of $X(A)$. 
The \emph{support of~$\nu$}, denoted by~$\supp(\nu)$, is the smallest closed subset~$C$ of~$X(A)$ for which $\nu(X(A)\setminus C)=0$ holds.\par\medskip

\noindent
We give a complete characterization of linear functionals that admit representing Radon measures with compact support in terms of a new growth condition (see~\eqref{thm::MP<->growth::cond}) intrinsic to the given linear functional. 
Since this growth condition implies Carleman's condition, we can use the recent general version of the classical Nussbaum theorem in~\cite[Theorem~3.17]{InKuKuMi20} to establish the existence of a unique representing Radon measure. 
The main novelty is that our growth condition surprisingly allows us to exactly identify the compact support of the representing Radon measure (see \eqref{thm::MP<->growth::eq1}). 
More precisely, we establish the following main result in Section~\ref{sec::char}.

\begin{thm}
\thlabel{thm::MP<->growth}
Let $L\colon A\to\RR$ be linear with~$L(A^2)\subseteq [0,\infty)$ and~$L(1)=1$.
Then there exists a unique representing Radon measure~$\nu_L$ for~$L$ with compact support if and only if 
\begin{equation}\label{thm::MP<->growth::cond}
\sup_{n\in\NN}\sqrt[2n]{L(a^{2n})}<\infty\text{ for all }a\in A.
\end{equation}
Moreover, in this case, 
\begin{equation}
\label{thm::MP<->growth::eq1}
\supp(\nu_L)=\{\alpha\in X(A):\left|\alpha(a)\right|\leq \sup_{n\in\NN}\sqrt[2n]{L(a^{2n})}\text{ for all }a\in A\}.
\end{equation}
\end{thm}

\noindent
To the best of our knowledge, \eqref{thm::MP<->growth::eq1} is the first exact and explicit description of the compact support of the representing Radon measure $\nu_L$. 
Indeed, linear functionals that admit representing Radon measures with compact support have already been studied, mostly in relation to their non-negativity on Archimedean quadratic modules (see~\cite{Jac01} and also~\cite{Cas84,Cha88,JaPr01,Mar03,Put93,Smu91}) and more recently in relation to their continuity w.r.t.\ submultiplicative seminorms (see~\cite{GhKuMa14} and also~\cite{BeCh76,GhKu13,GhKuSa13,GhMaWa14,Las13,LaNe07}). 
However, those results focus on the existence of a representing Radon measure while the support is only shown to be contained in a compact set associated with the considered quadratic module resp.\ submultiplicative seminorm. 
 
In Sections~\ref{sec::growth}, \ref{sec::continuity}, and~\ref{sec::positivity}, we analyze the equivalence of our growth condition~\eqref{thm::MP<->growth::cond}, the positivity condition in~\cite{Jac01}, and the continuity condition in~\cite{GhKuMa14} independently of the representing Radon measure $\nu_L$.
This structural analysis allows us, in Section~\ref{sec::support}, to establish \eqref{thm::MP<->growth::eq1} and to characterize $\supp(\nu_L)$ in terms of the largest Archimedean quadratic module on which~$L$ is non-negative as well as in terms of the smallest submultiplicative seminorm  w.r.t.\! which $L$ is continuous (see~\thref{char::compact::support}). 
For the convenience of the reader, we collect in \thref{char::compact} all the above mentioned equivalent conditions and the  characterizations of $\supp(\nu_L)$. 
In Section~\ref{sec::support}, we also present an explicit formula for computing the measure of singletons in $\supp(\nu_L)$ (see \thref{thm::thm}). 
From this result, we derive a sufficient condition for $\supp(\nu_L)$ to be countable (see~\thref{thm::thm::cor1}) as well as a
necessary and sufficient condition for $\supp(\nu_L)$ to be finite (see~\thref{thm::thm::cor2}). 

In Section~\ref{sec::natural}, we construct and compare two locally convex topologies on~$A$ closely related to the growth condition~\eqref{thm::MP<->growth::cond} and compatible with the algebraic structure of~$A$ (see~\thref{prop::can-top1,prop::can-top2}).
In Section~\ref{sec::generators}, we show that if $A$ is endowed with a locally convex topology belonging to a certain class, then
assuming the growth condition \eqref{thm::MP<->growth::eq1} only on the generating elements of a dense subalgebra of $A$ is sufficient for the existence of $\nu_L$ (see~\thref{thm-gen}).

\section{Preliminaries}

In this section we collect some fundamental concepts, notations, and results which we will repeatedly use in the following.\par\medskip

\noindent
Throughout this article $A$ denotes a unital commutative $\RR$--algebra with non-empty character space. Recall that the topology~$\tau_{X(A)}$ is Hausdorff and that the collection of sets of the form
\begin{equation}\label{eq::basis-weak-top}
U(a):=\{\alpha\in X(A): \hat{a}(\alpha)>0\}\qquad\text{with }a\in A
\end{equation}
is a basis of~$\tau_{X(A)}$ (see~\cite[Section~2.1]{InKuKuMi20} for details).\par\medskip 

A subset~$Q\subseteq A$ is a \emph{quadratic module (in~$A$)} if $1\in Q$, $Q+Q\subseteq Q$, and~$A^2Q\subseteq Q$.
The set $\sum A^2$ of all finite sums of squares of elements in~$A$ is the smallest quadratic module in~$A$. 
If in addition for each~$a\in A$ there exists~$N\in\NN$ such that~$N\pm a\in Q$, then $Q$ is \emph{Archimedean}.
The \emph{non-negativity set} of a quadratic module~$Q$ in~$A$ is defined as
$$
\nn(Q):=\{\alpha\in X(A): \hat{a}(\alpha)\geq 0\text{ for all }a\in Q\}\subseteq X(A), 
$$
which is closed. 
If~$Q$ is Archimedean, then~$\nn(Q)$ is compact (see, e.g.,~\cite[Theorem~5.7.2]{Mar08}) while the converse is false in general (see~\cite{JaPr01}).
Given~$C\subseteq X(A)$ closed, the set
$$
\Pos(C):=\{a\in A\colon \hat{a}(\alpha)\geq 0\text{ for all }\alpha\in C\}
$$
is a quadratic module, which is Archimedean if~$C$ is compact.

\begin{prop}
\thlabel{prop::nn-vs-Pos}
The following statements hold:
\begin{enumerate}[(i)]
	\item\label{prop::nn-vs-Pos::1}
	$\nn(\Pos(C))=C$ for all closed $C\subseteq X(A)$.
	\item\label{prop::nn-vs-Pos::2}
	$Q\subseteq\Pos(\nn(Q))$ for all quadratic modules~$Q$ in~$A$.
\end{enumerate}
\end{prop}
\begin{proof}
For \eqref{prop::nn-vs-Pos::1} let~$C\subseteq X(A)$ be closed. Let~$\beta\in\nn(\Pos(C))$. 
Since~$C$ is closed w.r.t.~$\tau_{X(A)}$, it suffices to show that $U(a)\cap C\neq\emptyset$ for all~$a\in A$ with $\beta\in U(a)$ (cf.~\eqref{eq::basis-weak-top}). 
Now, let~$a\in A$ such that $\beta\in U(a)$ and assume for a contradiction that $U(a)\cap C=\emptyset$. 
Then $\hat{a}(\alpha)\leq 0$ for all~$\alpha\in C$, i.e., $-a\in\Pos(C)$, and so~$\beta(-a)\geq 0$. 
This contradicts $\beta\in U(a)$, i.e., $U(a)\cap C\neq\emptyset$, and hence,~$\beta\in C$.

The converse inclusion in~\eqref{prop::nn-vs-Pos::1} and statement~\eqref{prop::nn-vs-Pos::2} are easy to verify.
\end{proof}

\noindent
Throughout this article each linear functional $L\colon A\to\RR$ is assumed to be \emph{normalized}, that is, $L(1)=1$.
A linear functional $L\colon A\to\RR$ is \emph{$Q$--positive} on a quadratic module~$Q$ in~$A$ if $L(Q)\subseteq[0,\infty)$.

\begin{lem}
\thlabel{lem::char-posi}
Let $L\colon A\to\RR$ be a linear functional and let~$Q$ be an Archimedean quadratic module in~$A$. 
If~$L$ is $Q$--positive, then~$L$ is $\Pos(\nn(Q))$--positive.
\end{lem}
\begin{proof}
Let~$L$ be $Q$--positive and~$a\in\Pos(\nn(Q))$. 
Then, for each~$\varepsilon>0$, the Jacobi Positiv\-stellensatz (see \cite[Theorem~4]{Jac01}) implies that $a+\varepsilon\in Q$ and so $L(a)+\varepsilon=L(a+\varepsilon)\geq 0$, i.e., $L(a)\geq 0$ as~$\varepsilon>0$ was arbitrary. 
\end{proof}

A function $p\colon A\to[0,\infty)$ is a \emph{seminorm (on~$A$)} if $p(\lambda a)=\left|\lambda\right|p(a)$ and~$p(a+b)\leq p(a)+p(b)$ for all~$\lambda\in\RR$ and all~$a,b\in A$. 
If in addition $p(ab)\leq p(a)p(b)$ for all~$a,b\in A$, then~$p$ is \emph{submultiplicative}. A linear functional $L\colon A\to\RR$ is \emph{$p$--continuous} w.r.t.\ a seminorm~$p$ on~$A$ if there exists~$C>0$ such that $\left|L(a)\right|\leq Cp(a)$ for all~$a\in A$.
The \emph{Gelfand spectrum} of a seminorm~$p$ on~$A$ is defined as 
$$
\sp(p):=\{\alpha\in X(A):\alpha\text{ is }p\text{--continuous}\}\subseteq X(A),
$$
and is $\sigma$--compact. 
If~$p$ is submultiplicative, then the Gelfand spectrum equals
\begin{equation}\label{eq::char-Gs}
\sp(p)=\{\alpha\in X(A):\left|\alpha(a)\right|\leq 1\cdot p(a)\text{ for all }a\in A\}
\end{equation}
and is compact (see~\cite[Lemma~3.2 and Corollary~3.3]{GhKuMa14}).
Viceversa, given~$K\subseteq X(A)$ compact, the function $\|\cdot\|_K\colon A\to[0,\infty)$ defined by
$$
\|a\|_K:=\max_{\alpha\in K}\left|\alpha(a)\right|<\infty\qquad\text{for all }a\in A
$$
is a submultiplicative seminorm.

\begin{prop}
\thlabel{prop::sp-vs-max}
The following statements hold:
\begin{enumerate}[(i)]
	\item\label{prop::sp-vs-max::1}
	$\sp(\|\cdot\|_K)=K$ for all compact~$K\subseteq X(A)$.
	\item\label{prop::sp-vs-max::2}
	$\|\cdot\|_{\sp(p)}\leq p$ for all submultiplicative seminorms~$p$ on~$A$.
\end{enumerate}
\end{prop}
\begin{proof}
For~\eqref{prop::sp-vs-max::1} let~$K\subseteq X(A)$ be compact.
Let~$\beta\in\sp(\|\cdot\|_K)$. 
Since~$K$ is closed w.r.t.\ $\tau_{X(A)}$, it suffices to show that $U(a)\cap K\neq\emptyset$ for all~$a\in A$ with $\beta\in U(a)$ (cf.\ \eqref{eq::basis-weak-top}).
Now, let~$a\in A$ be such that $\beta\in U(a)$ and set $b:=a+\|a\|_K$. 
Then $\beta(b)=\beta(a)+\|a\|_K>\|a\|_K$ as $\beta\in U(a)$ and also $\alpha(b)=\alpha(a)+\|a\|_K\geq 0$ for all~$\alpha\in K$ as $\left|\alpha(a)\right|\leq\|a\|_K$ by definition. 
Since $\beta\in\sp(\|\cdot\|_K)$ this implies that 
$$
\|\alpha\|_K<\beta(b)\leq \|b\|_K=\max_{\alpha\in K}\alpha(a)+\|a\|_K.
$$
Therefore, there exists $\alpha\in K$ with~$\alpha(a)>0$, i.e., $U(a)\cap K\neq\emptyset$, and hence,~$\beta\in K$.

The converse inclusion in~\eqref{prop::sp-vs-max::1} and statement~\eqref{prop::sp-vs-max::2} are easy to verify.
\end{proof}

\noindent
Each $\sum A^2$--positive linear functional $L\colon A\to\RR$ satisfies the Cauchy--Bunyakovsky--Schwarz inequality, i.e.,
\begin{equation}\label{eq::CBS}\tag{CBS}
L(ab)^2\leq L(a^2)L(b^2)\qquad\text{for all }a,b\in A.
\end{equation}
Repeated application of~\eqref{eq::CBS} yields that~$L$ is continuous w.r.t.\ a submultiplicative seminorm~$p$ on~$A$ if and only if
\begin{equation}\label{lem::growth->subm::rem1::eq1}
\left|L(a)\right|\leq 1\cdot p(a)\qquad\text{for all }a\in A.
\end{equation}

\noindent
This combined with a result in the theory of complex Banach algebras leads to the following result. 
Recall that a \emph{complex Banach algebra} is a pair~$(B,q)$ consisting of a $\CC$--al\-ge\-bra~$B$ and a submultiplicative norm~$q$ on~$B$ such that the topology on~$B$ generated by~$q$ is complete.

\begin{lem}
\thlabel{lem::char-conti}
Let $L\colon A\to\RR$ be a $\sum A^2$--positive linear functional and let~$p$ be a submultiplicative seminorm on~$A$. 
If~$L$ is $p$--continuous, then~$L$ is $\|\cdot\|_{\sp(p)}$--continuous.
\end{lem}
\begin{proof}
Passing to the completion of the canonical quotient of the seminormed algebra~$(A,p)$ (see~\cite[Remark~3.4]{GhKuMa14}) and then to its complexification (see~\cite[I, \S13, Proposition~3 (p.~68)]{BoDu73}), there exists a complex Banach algebra~$(B,q)$ and a homomorphism $\phi\colon A\to B$ such that $p=q\circ\phi$.
Now, let~$L$ be $p$--continuous and~$a\in A$. 
Then $L(a)^{2^d}\leq L(a^{2^d})\leq 1\cdot p(a^{2^d})=q(\phi(a^{2^d}))=q(\phi(a)^{2^d})$ for all~$d\in\NN_0$ by~\eqref{lem::growth->subm::rem1::eq1}, and so~\cite[VII, Theorem~8.9 (p.~220)]{Co90} implies that 
$$
\left|L(a)\right|\leq\lim_{d\to\infty}\sqrt[2^d]{q(\phi(a)^{2^d})}=\|\phi(a)\|_{\sp(q)},
$$
i.e., $\left|L(a)\right|\leq\|a\|_{\sp(p)}$ as $\|\phi(a)\|_{\sp(q)}\leq \|a\|_{\sp(p)}$.
\end{proof}

\section{Main results}\label{sec::char}

In this section we first give a proof of~\thref{thm::MP<->growth} using the recent general version of the classical Nussbaum theorem in~\cite[Theorem~3.17]{InKuKuMi20}.
Then we relate~\thref{thm::MP<->growth} to the characterizations of linear functionals that admit representing Radon measures with compact support in~\cite{GhKuMa14} and~\cite{Jac01} (see~\thref{thm::MP<->subm,thm::MP<->Arch}, respectively).
This is summarized in~\thref{char::compact}, where we also provide a detailed analysis of the compact support.

\begin{proof}[Proof of~\thref{thm::MP<->growth}]
Set $C_a:=\sup_{n\in\NN}\sqrt[2n]{L(a^{2n})}<\infty$ for all~$a\in A$. 
In order to apply~\cite[Theorem~3.17]{InKuKuMi20} we introduce the index set
$$
I:=\{S\subseteq A: S\text{ finitely generated (unital) subalgebra of }A\}
$$
and note that, for each~$a\in A$, clearly $\sum_{n=1}^\infty(\sqrt[2n]{L(a^{2n})})^{-1}=\infty$. 
Therefore, by~\cite[Theorem~3.17-(i)]{InKuKuMi20}, for each~$S\in I$ there exists a unique representing Radon measure~$\nu_S$ on~$X(S)$ for~$L\!\restriction_S$.
Let us now show that the family~$\{\nu_S:S\in I\}$ fulfils the so-called Prokhorov condition by means of the characterization in~\cite[Proposition~1.18]{InKuKuMi20}. 
For each~$S\in I$, define $K^{(S)}:=\{\alpha\in X(S):\left|\alpha(a)\right|\leq C_a\text{ for all }a\in S\}$. 
Now, let~$S\subseteq T$ in~$I$. 
The closed set $K^{(S)}\subseteq X(S)$ is compact as it embeds into the compact product $\prod_{a\in S}[-C_a,C_a]$ via the continuous map $\alpha\mapsto (\alpha(a))_{a\in S}$ and the inclusion $\pi_{S,T}(K^{(T)})\subseteq K^{(S)}$ holds by definition, where $\pi_{S,T}\colon X(T)\to X(S)$ denotes the restriction map. 
To show that $\nu_S(K^{(S)})=1$, let~$a\in S$ as well as $\varepsilon>0$ and consider the set $B(a,\varepsilon):=\{\alpha\in X(S):\left|\alpha(a)\right|\leq C_a+\varepsilon\}$. 
Then, for each~$n\in\NN$, Chebyshev's inequality implies that
$$
\nu_S(X(S)\setminus B(a,\varepsilon))\cdot(C_a+\varepsilon)^{2n} \leq\int\hat{a}^{2n}\dd\nu_S =L(a^{2n})\leq C_a^{2n}
$$
and so $\nu_S(B(a,\varepsilon))=1$ as~$C_a(C_a+\varepsilon)^{-1}<1$ and~$n\in\NN$ was arbitrary. 
Since the equality
$$
K^{(S)}=\bigcap_{a\in S}\{\alpha\in X(S):\left|\alpha(a)\right|\leq C_a\}=\bigcap_{a\in S}\bigcap_{\varepsilon>0}B(a,\varepsilon)
$$
holds, \cite[I, \S6 (a) (p.~40)]{Swa73} yields that $\nu_S(K^{(S)})=1$ as~$\nu_S$ is Radon and $B(a,\varepsilon)$ is closed. 
Therefore, the family~$\{\nu_S:S\in I\}$ fulfils Prokhorov's condition and hence, by~\cite[Theorem~3.17-(iii)]{InKuKuMi20} there exists a unique representing Radon measure~$\nu$ for~$L$. Since~$\nu$ is a Radon measure, we can argue as for the set~$K^{(S)}$ to show that $\nu(K_L)=1$, where 
\begin{equation}\label{thm::MP<->growth::eq2}
K_L:=\{\alpha\in X(A): |\alpha(a)|\leq C_a\text{ for all }a\in A\}.
\end{equation}
This establishes~$\supp(\nu)\subseteq K_L$, while~$\supp(\nu)\supseteq K_L$ is established in \thref{char::compact::support}.
Moreover, $K_L$ is compact as it embeds into~$\prod_{a\in A}[-C_a,C_a]$ and so $\supp(\nu)$ is compact, too.

Conversely, let~$\nu_L$ be the representing Radon measure for~$L$ with compact support and let~$a\in A$. 
Then
$$
L(a^{2n})=\int\hat{a}^{2n}\dd\nu_L\leq\max_{\alpha\in\supp(\nu_L)}\left|\alpha(a)\right|^{2n}<\infty\qquad\text{for all }n\in\NN
$$
shows that $\sup_{n\in\NN}\sqrt[2n]{L(a^{2n})}<\infty$.
\end{proof}
 
\begin{cor}
\thlabel{char::compact}
For a $\sum A^2$--positive linear functional $L\colon A\to\RR$ the following are equivalent:
\begin{enumerate}[(i)]
	\item\label{char::compact::1}
	There exists a unique representing Radon measure~$\nu_L$ for~$L$ with compact support.
	\item\label{char::compact::2}
	$\sup_{n\in\NN}\sqrt[2n]{L(a^{2n})}<\infty$ for all~$a\in A$.
	\item\label{char::compact::3}
	$L$ is $p$--continuous for some submultiplicative seminorm~$p$ on~$A$.
	\item\label{char::compact::4}
	$L$ is $Q$--positive for some Archimedean quadratic module~$Q$ in~$A$.	
\end{enumerate}
In this case, the submultiplicative seminorm defined by $p_L(a):=\sup_{n\in\NN}\sqrt[2n]{L(a^{2n})}$ for all~$a\in A$ (see~\eqref{lem::growth->subm::eq1}) is the smallest on~$A$ w.r.t.\ which~$L$ is continuous and the Archimedean quadratic module generated by $\sup_{n\in\NN}\sqrt[2n]{L(a^{2n})}\pm a$ with $a\in A$ (see~\eqref{lem::growth->Arch::eq1}) is the largest in~$A$ on which~$L$ is positive. 
Moreover, $p_L=\|\cdot\|_{\supp(\nu_L)}$ and $Q_L=\Pos(\supp(\nu_L))$ as well as
\begin{equation}\label{loc-supp}
\supp(\nu_L)=K_L=\sp(p_L)=\nn(Q_L),
\end{equation}
where~$K_L$ is defined in~\eqref{thm::MP<->growth::eq2}.
\end{cor}

\noindent
If the $\sum A^2$--positive linear functional $L\colon A\to\RR$ is represented by the Radon measure~$\nu_L$ with compact support, then it is easy to show that~$L$ is $\|\cdot\|_{\supp(\nu_L)}$--continuous and $\Pos(\supp(\nu_L))$--positive. 
This yields the implications~\eqref{char::compact::1}$\Rightarrow$\eqref{char::compact::3} and \eqref{char::compact::1}$\Rightarrow$\eqref{char::compact::4}, respectively. 
All other remaining implications are shown in the following subsections as illustrated by the diagram.
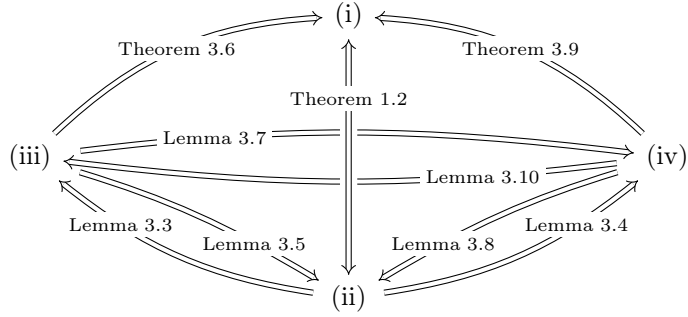
\begin{center}
\begin{tikzcd}
&&&& \eqref{char::compact::1}
\arrow[lllldd, Leftarrow, bend right=20, "\text{\thref{thm::MP<->subm}}"' description]
\arrow[rrrrdd, Leftarrow, bend left=20, "\text{\thref{thm::MP<->Arch}}" description]
&&&&\\
&&&& &&&&\\
\eqref{char::compact::3}
\arrow[rrrrrrrr, Rightarrow, bend left=7, "\text{\thref{lem::subm->Arch}}" description, near start, shorten <=1.5ex]
\arrow[rrrrdd, Rightarrow, bend left=7, "\text{\thref{lem::subm->growth}}" description, near end, shorten <=1.5ex]
&&&& &&&&
\eqref{char::compact::4}
\arrow[llllllll, Rightarrow, bend left=7, "\text{\thref{lem::Arch->subm}}" description, near start, shorten <=1.5ex]
\arrow[lllldd, Rightarrow, bend right=7, "\text{\thref{lem::Arch->growth}}" description, near end, shorten <=1.5ex]
\\
&&&& &&&&\\
&&&& \eqref{char::compact::2}
\arrow[uuuu, Leftrightarrow, crossing over, "\text{\thref{thm::MP<->growth}}" description, near end]
\arrow[lllluu, Rightarrow, bend left=15, "\text{\thref{lem::growth->subm}}" description, near end, shorten <=.5ex]
\arrow[rrrruu, Rightarrow, bend right=15, "\text{\thref{lem::growth->Arch}}" description, near end, shorten <=.5ex]
&&&&\\
\end{tikzcd}
\end{center}\vspace{-\baselineskip}
In fact, we establish the equivalences of the growth condition~\eqref{char::compact::2} (see \eqref{thm::MP<->growth::cond}), the continuity condition~\eqref{char::compact::3}, and the positivity condition~\eqref{char::compact::4} without appealing to the representing Radon measure $\nu_L$.
The localization of the support in~\eqref{loc-supp} is shown in~\thref{char::compact::support}.

\subsection{The growth condition}\label{sec::growth}
 
In the following we first analyze in detail some properties of the growth condition \eqref{thm::MP<->growth::cond} (i.e., \thref{char::compact}-\eqref{char::compact::2}) and then we directly derive from it \thref{char::compact}-\eqref{char::compact::3}~and \thref{char::compact}-\eqref{char::compact::4}.

\begin{rem}
\thlabel{thm::main-thm::rem1}
Let $L\colon A\to\RR$ be a $\sum A^2$--positive linear functional.
\begin{enumerate}[(i)]
	\item\label{rem::sup-power-generators::1}
	Then~\eqref{eq::CBS} yields that for each~$a\in A$ the sequence $(\sqrt[2n]{L(a^{2n})})_{n\in\NN}$ is monotone increasing. 
	Therefore, the growth condition \eqref{thm::MP<->growth::cond} is of asymptotic nature, i.e.,
	$$
	\sup_{n\in\NN}\sqrt[2n]{L(a^{2n})}=\lim_{n\to\infty}\sqrt[2n]{L(a^{2n})}\qquad\text{for all }a\in A,
	$$
	which allows us to work with suitable subsequences. 
	For example the growth condition \eqref{thm::MP<->growth::cond} holds if and only if for each $a\in A$ there exists $f_a\colon\NN\to\NN$ unbounded, $g_a\colon\NN\to[0,\infty)$ with $\lim_{n\to\infty}\sqrt[2f_a(n)]{g_a(n)}<\infty$, and $C_a\geq 0$ such that $L(a^{2f_a(n)})\leq g_a(n)C_a^{2f_a(n)}$ for all $n\in\NN$.
	
	\item\label{rem::sup-power-generators::0} 
	The subsequence $(\sqrt[2^d]{L(a^{2^d})})_{d\in\NN}$ will be crucial as for each $d\in\NN$ setting 
	\begin{equation}
	p_d(a):=\sqrt[2^d]{L(a^{2^d})}\qquad\text{for all }a\in A
	\end{equation}
	defines a seminorm. 
	Indeed, let~$d\in\NN$ and note that $p_d(ab)\leq p_{d+1}(a)p_{d+1}(b)$ for all~$a,b\in A$ by~\eqref{eq::CBS} as $p_{d+1}(a)^2=p_d(a^2)$ for all~$a\in A$ by definition. 
	To show that~$p_d$ is a seminorm, we proceed by induction on $d$. Since $(a,b)\mapsto L(ab)$ defines a positive semidefinite bilinear form, $p_1$ is a seminorm. 
	Now, assume that~$p_d$ is a seminorm and let~$\lambda\in\RR$ and~$a,b\in A$. 
	Then $p_{d+1}(\lambda a)^2=\lambda^2p_d(a^2)$ and 
	$$
	p_d(a^2)+2p_d(ab)+p_d(b^2)\leq p_{d+1}(a)^2+2p_{d+1}(a)p_{d+1}(b)+p_{d+1}(b)^2,
	$$
	i.e., $p_{d+1}(a+b)^2\leq (p_{d+1}(a)+p_{d+1}(b))^2$,
	yield that~$p_{d+1}$ is a seminorm. 
	Moreover, note that $\ker(p_d)=\ker(p_1)$ as \eqref{eq::CBS} yields that $p_1(a)\leq p_d(a)$ and $p_d(a)^{2^d}\leq p_1(a)p_1(a^{2^d-1})$ for all $a\in A$.

	\item 
	If $L$ is represented by the Radon measure $\nu_L$ with compact support, then for each~$n\in\NN$ the function $a \mapsto \sqrt[2n]{L(a^{2n})}$ coincides with the $\L^{2n}$--seminorm of the Lebesgue space $\mathcal{L}^{2n}(X(A),\nu_L)$ of $2n$--integrable functions. 
	By a standard result of measure theory (cf.\ \cite[p.~143]{Koe69}), we get that $\sup_{n\in\NN}\sqrt[2n]{L(a^{2n})} = \|\hat{a}\|_\infty$ for all $a\in A$,	where~$\|\cdot\|_\infty$ denotes the (submultiplicative) $\L^\infty$--seminorm, given by the essential supremum on~$X(A)$, of $\mathcal{L}^\infty(X(A),\nu_L)$.
	\end{enumerate}
\end{rem}

\begin{lem}
\thlabel{lem::growth->subm}
Let $L\colon A\to\RR$ be a $\sum A^2$--positive linear functional on $A$ such that $\sup_{n\in\NN}\sqrt[2n]{L(a^{2n})}<\infty$ for all~$a\in A$. 
Then setting
\begin{equation}\label{lem::growth->subm::eq1}
p_L(a):=\sup_{n\in\NN}\sqrt[2n]{L(a^{2n})}\qquad\text{for all }a\in A
\end{equation}
defines a submultiplicative seminorm w.r.t.\ which~$L$ is continuous.
\end{lem}
\begin{proof}
Recall that $\sup_{n\in\NN}\sqrt[2n]{L(a^{2n})}=\sup_{d\in\NN}p_d(a)$ for all~$a\in A$ by~\thref{thm::main-thm::rem1}-\eqref{rem::sup-power-generators::1}. 
Thus, using~\thref{thm::main-thm::rem1}-\eqref{rem::sup-power-generators::0}, it is easy to verify that~$p_L$ is a seminorm which is submultiplicative as 
$$
p_L(ab)=\sup_{d\in\NN}p_d(ab)\leq \sup_{d\in\NN}p_{d+1}(a)\cdot\sup_{d\in\NN}p_{d+1}(b)=p_L(a)p_L(b)
$$
for all~$a,b\in A$. 
Clearly, $L$ is $p_L$--continuous as $\left|L(a)\right|\leq \sqrt{L(a^2)}\leq 1\cdot p_L(a)$ for all~$a\in A$ by~\eqref{eq::CBS}.
\end{proof}

\noindent
Note that $\sp(p_L)=K_L$ by~\eqref{eq::char-Gs} as~$p_L$ is submultiplicative.

\begin{lem}
\thlabel{lem::growth->Arch}
Let $L\colon A\to\RR$ be a $\sum A^2$--positive linear functional on~$A$ such that $\sup_{n\in\NN}\sqrt[2n]{L(a^{2n})}<\infty$ for all~$a\in A$. 
Then
\begin{equation}\label{lem::growth->Arch::eq1}
Q_L := \{a\in A:L(b^2a)\geq 0\text{ for all }b\in A\}
\end{equation}
is an Archimedean quadratic module on which~$L$ is positive. 
In fact, $Q_L$ is generated by~$\sup_{n\in\NN}\sqrt[2n]{L(a^{2n})}\pm a$ with~$a\in A$.
\end{lem}
\begin{proof}
It is easy to verify that~$Q_L$ is a quadratic module and~$L$ is $Q_L$--positive as $L(a)=L(1^2\cdot a)\geq 0$ for all~$a\in Q_L$. 
It remains to show that~$Q_L$ is Archimedean.
Let~$a,b\in A$ such that, w.l.o.g., $L(b^2)=1$ as well as $d\in\NN$ and recall that $C_a:=\sup_{n\in\NN}\sqrt[2n]{L(a^{2n})}=\sup_{d\in\NN}\sqrt[2^d]{L(a^{2^d})}$ by~\thref{thm::main-thm::rem1}-\eqref{rem::sup-power-generators::1}.
Then $L(a^{2^{d+1}})\leq C_a^{2^{d+1}}$ by definition and $\left|L(b\cdot ba)\right|^{2^{d+1}}\leq 1\cdot L(b^2a^{2^d})^2\leq L(b^4)L(a^{2^{d+1}})$ by repeated application of~\eqref{eq::CBS}.
Therefore, 
$$
L(b^2(C_a\pm a))= C_aL(b^2)\pm L(b^2a)\geq C_a(1-\sqrt[2^{d+1}]{L(b^4)})
$$
and so $L(b^2(C_a\pm a))\geq 0$ as~$d\in\NN$ was arbitrary, i.e., $C_a\pm a\in Q_L$.

Now, let $a\in Q_L$. 
Then $C_a+(a-C_a)=a$ and $C_a-(a-C_a)=C_a+(C_a-a)$ are in~$Q_L$ and so $C_a^2-(a-C_a)^2\in Q_L$ by \cite[Proposition~5.2.3-(1)]{Mar08}. 
Therefore, an easy induction shows that $C_a^{2^d}\pm (a-C_a)^{2^d}\in Q_L$ for all~$d\in\NN$ and so
$$
C_{a-C_a}:=\sup_{n\in\NN}\sqrt[2n]{L((a-C_a)^{2n})}=\sup_{d\in\NN}\sqrt[2^d]{L((a-C_a)^{2^d})}\leq C_a
$$
as~$L$ is $Q_L$--positive. Hence, the identity $a=(C_a-C_{a-C_a})+(C_{a-C_a}+(a-C_a))$ shows that~$Q_L$ is generated by~$C_a\pm a$ with $a\in A$.
\end{proof}

\noindent
Note that $\nn(Q_L)=K_L$ as ~$Q_L$ is generated by~$\sup_{n\in\NN}\sqrt[2n]{L(a^{2n})}\pm a$ with $a\in A$.

\subsection{The continuity condition}\label{sec::continuity}

In the following, we directly derive from the continuity condition \thref{char::compact}-\eqref{char::compact::3} all other conditions in~\thref{char::compact}.

\begin{lem}
\thlabel{lem::subm->growth}
Let $L\colon A\to\RR$ be a $\sum A^2$--positive linear functional  that is $p$--continuous for a submultiplicative seminorm~$p$ on~$A$.
Then $\sup_{n\in\NN}\sqrt[2n]{L(a^{2n})}<\infty$ for all~$a\in A$.
In particular, $p_L\leq p$.
\end{lem}
\begin{proof}
Let~$a\in A$. Then $L(a^{2n})\leq 1\cdot p(a)^{2n}$ for all~$n\in\NN$ by the sub\-mul\-ti\-pli\-ca\-tivity of~$p$ and \eqref{lem::growth->subm::rem1::eq1}, i.e., $p_L(a)=\sup_{n\in\NN}\sqrt[2n]{L(a^{2n})}\leq p(a)<\infty$.
\end{proof}

\noindent
The implication~\eqref{char::compact::3}$\Rightarrow$\eqref{char::compact::1} in~\thref{char::compact} follows from~\thref{thm::MP<->subm} below. 
\thref{thm::MP<->subm} was established in~\cite[Corollary~3.8]{GhKuMa14} using the well-known Jacobi Positivstellensatz (see~\cite[Theorem~4]{Jac01}) and a result from the theory of \emph{real} Banach algebras (see~\cite[Lemma~3.5]{GhKuMa14}). 
We provide an alternative proof that does not involve any Positivstellensatz but instead relies on the functional calculus for \emph{complex} Banach algebras.
Note that~\thref{thm::MP<->subm} also follows from~\thref{thm::MP<->growth} without appealing to neither any Positivstellensatz nor any theory of (real or complex) Banach algebras. 
Indeed, combining~\thref{lem::subm->growth,thm::MP<->growth} yields that~$\nu_L$ is the representing Radon measure for~$L$ and $\supp(\nu_L)\subseteq K_L=\sp(p_L)\subseteq\sp(p)$.

\begin{thm}
\thlabel{thm::MP<->subm}
Let $L\colon A\to\RR$ be a $\sum A^2$--positive linear functional  that is $p$--continuous for a submultiplicative seminorm~$p$ on~$A$.
Then there exists a re\-pre\-senting Radon measure~$\nu$ for~$L$ with $\supp(\nu)\subseteq\sp(p)$.
\end{thm}
\begin{proof}
As in the proof of~\thref{lem::char-conti} there exists a complex Banach algebra~$(B,q)$ and a homomorphism~$\phi\colon A\to B$ such that~$p=q\circ\phi$. 
By construction, for each $\beta\in\sp(q)\subseteq X(B)$ there exists $\alpha\in\sp(p)\subseteq X(A)$ such that $\beta\circ\phi=\alpha$. 
In particular, since~$L$ is $p$--continuous, there exists a unique linear functional $\overline{L}\colon B\to\CC$ such that $L=\overline{L}\circ\phi$.
By construction, $\overline{L}$ is non-negative on Hermitian squares of~$B$.

Now, let~$a\in\Pos(\sp(p))$ and $\varepsilon>0$. 
Then $\phi(a)\in\Pos(\sp(q))$ by construction and the spectrum of $\phi(a)+\varepsilon$ in the sense of complex Banach algebras (see~\cite[II, \S16, Proposition~9 (p.~81)]{BoDu73}), i.e., $\{\beta(\phi(a))+\varepsilon : \beta\in\sp(q)\}$, is contained in~$[\varepsilon,\infty)$. 
Therefore, by the functional calculus for complex Banach algebras (cf.\ \cite[I, \S7, Theorem~4 (p.~33)]{BoDu73}), there exists Hermitian~$b\in B$ such that $\phi(a)+\varepsilon=b^2$ and so 
$$
L(a)+\varepsilon = L(a+\varepsilon) = \overline{L}(\phi(a)+\varepsilon) = \overline{L}(b^2)\geq 0,
$$
i.e., $L(a)\geq 0$ as~$\varepsilon>0$ was arbitrary. 
Thus, $L$ is $\Pos(\sp(p))$--positive and so, by~\cite[Theorem~3.3.2]{Mar08}, there exists representing Radon measure~$\nu$ for~$L$ with $\supp(\nu)\subseteq\sp(p)$. Recall that $\sp(p)\subseteq X(A)$ is compact as~$p$ is submultiplicative.
\end{proof}

\begin{lem}
\thlabel{lem::subm->Arch}
Let $L\colon A\to\RR$ be a $\sum A^2$--positive linear functional  that is $p$--continuous for a submultiplicative seminorm~$p$ on~$A$.
Then $L$ is $\Pos(\sp(p))$--positive.
\end{lem}
\begin{proof}
Set $K:=\sp(p)$. Recall that~$\Pos(K)$ is Archimedean as~$K$ is compact and that~$L$ is $\|\cdot\|_K$--continuous by~\thref{lem::char-conti}. 
Let~$a\in\Pos(K)$. 
Then $\|a-\|a\|_K\|_K\leq\|a\|_K$ as $\left|\alpha(a-\|a\|_K)\right|\leq \|a\|_K$ for all~$\alpha\in K$ and so $\left|L(a-\|a\|_K)\right|\leq 1\cdot\|a-\|a\|_K\|_K$ by~\eqref{lem::growth->subm::rem1::eq1}. 
Therefore,
$$
\|a\|_K + L(a-\|a\|_K) \geq \|a\|_K - \|a-\|a\|_K\|_K\geq 0,
$$
i.e., $L(a)\geq 0$.
\end{proof}

\noindent
Given a $\sum A^2$--positive linear functional $L\colon A\to\RR$ that is $p$--continuous for a submultiplicative seminorm~$p$ on~$A$, it is also possible to show that~$L$ positive on the (Archimedean) quadratic module generated by~$p(a)\pm a$ with~$a\in A$ (cf.\ \thref{lem::growth->Arch}).
The closure of this quadratic module w.r.t.\ the finest locally convex topology on~$A$ is equal to~$\Pos(\sp(p))$ as a consequence of the Jacobi Positiv\-stellensatz (see~\cite[Theorem~4]{Jac01}).
It is not clear to us if the quadratic module itself is equal to~$\Pos(\sp(p))$.\par\medskip

\subsection{The positivity condition}\label{sec::positivity}

In the following, we directly derive from the positivity condition~\thref{char::compact}-\eqref{char::compact::4} all other conditions in~\thref{char::compact}.

\begin{lem}
\thlabel{lem::Arch->growth}
Let $L\colon A\to\RR$ be a linear functional that is $Q$--positive for an Archimedean quadratic module~$Q$ in~$A$.
Then $\sup_{n\in\NN}\sqrt[2n]{L(a^{2n})}<\infty$ for all~$a\in A$.
In particular, $Q\subseteq Q_L$.
\end{lem}
\begin{proof}
Let~$a\in A$. Then there exists~$N\in\NN$ such that $N\pm a\in Q$ by the Archimedianity of~$Q$ and so $N^{2^d}\pm a^{2^d}\in Q$ for all $d\in\NN$ as in the proof of \thref{lem::growth->Arch}. Therefore, $\sup_{n\in\NN}\sqrt[2n]{L(a^{2n})}=\sup_{d\in\NN}\sqrt[2^d]{L(a^{2^d})}\leq N<\infty$ by \thref{thm::main-thm::rem1}-\eqref{rem::sup-power-generators::1} as~$L$ is $Q$--positive. 

Now, let~$a\in Q$. Then $L(b^2a)\geq 0$ for all~$b\in A$ as~$b^2a\in Q$, i.e., $a\in Q_L$.
\end{proof}

\noindent
The implication~\eqref{char::compact::4}$\Rightarrow$\eqref{char::compact::1} in~\thref{char::compact} follows from~\thref{thm::MP<->Arch} below. 
\thref{thm::MP<->Arch}~was established in, e.g.,~\cite[Corollary~3.3]{Mar03} using the Jacobi Positivstellensatz (see~\cite[Theorem~4]{Jac01}). 
Note that~\thref{thm::MP<->Arch} also follows from~\thref{thm::MP<->growth} without appealing to any Positivstellensatz. 
Indeed, combining~\thref{lem::Arch->growth,thm::MP<->growth} yields that~$\nu_L$ is the representing Radon measure for~$L$ and $\supp(\nu_L)\subseteq K_L=\nn(Q_L)\subseteq\nn(Q)$.

\begin{thm}
\thlabel{thm::MP<->Arch}
Let $L\colon A\to\RR$ be a linear functional that is $Q$--positive for an Archimedean quadratic module~$Q$ in~$A$.
Then there exists a representing Radon measure~$\nu$ for~$L$ with $\supp(\nu)\subseteq\nn(Q)$.
\end{thm}

\begin{lem}
\thlabel{lem::Arch->subm}
Let $L\colon A\to\RR$ be a linear functional that is $Q$--positive for an Archimedean quadratic module~$Q$ in~$A$.
Then $L$ is $\|\cdot\|_{\nn(Q)}$--continuous.
\end{lem}
\begin{proof}
Set $K:=\nn(Q)$. 
Recall that $\|\cdot\|_K$ is submultiplicative as~$K$ is compact and that~$L$ is $\Pos(K)$--positive by~\thref{lem::char-posi}.
Let~$a\in A$. 
Then $\|a\|_K\pm a\in\Pos(K)$ and so $L(\|a\|_K\pm a)\geq 0$, i.e, $\left|L(a)\right|\leq\|a\|_K$.
\end{proof}

\noindent
Given a linear functional $L\colon A\to\RR$ that is $Q$--positive for an Archimedean quadratic module~$Q$ in~$A$, is also possible to show that~$L$ is continuous w.r.t.\ the submultiplicative seminorm given by~$a\mapsto\inf\{s\geq 0:s\pm a\in Q\}$ (cf.\ \cite[Theorem~10.5]{Smu20}).
This seminorm is equal to~$\|\cdot\|_{\nn(Q)}$ as a consequence of the Jacobi Positiv\-stellensatz (see~\cite[Theorem~4]{Jac01}).\par\medskip

\subsection{Localization of the support}\label{sec::support}

In the following, we analyze in detail the support of the representing Radon measure in~\thref{char::compact}-\eqref{char::compact::1}.

\begin{cor}
\thlabel{char::compact::support}
Let $L\colon A\to\RR$ be represented by the Radon measure $\nu_L$ with compact support~$K$. 
Then $\|\cdot\|_K=p_L$ and $\Pos(K)=Q_L$ as well as
$$
K=K_L=\sp(p_L)=\nn(Q_L).
$$
\end{cor}
\begin{proof}
The seminorm~$p_L$ is well-defined by~\thref{thm::MP<->growth,lem::growth->subm}.
Further, $K\subseteq\sp(p_L)$ by~\thref{thm::MP<->subm}. 
Therefore, $\|\cdot\|_K\leq p_L$ by~\thref{prop::sp-vs-max}-\eqref{prop::sp-vs-max::2}. 
Since~$\nu_L$ is representing for~$L$, clearly, $L$ is $\|\cdot\|_K$--continuous and so $p_L\leq\|\cdot\|_K$ by~\thref{lem::subm->growth}. 
Hence, $\|\cdot\|_K=p_L$.
Analogously, $\Pos(K)=Q_L$.

Recall that $K_L=\sp(p_L)=\nn(Q_L)$ and note that $\sp(\|\cdot\|_K)=\sp(p_L)$ as well as $\nn(\Pos(K))=\nn(Q_L)$ by the first part of the proof. This together with~Propositions~\ref{prop::nn-vs-Pos}-\eqref{prop::nn-vs-Pos::1} and~\ref{prop::sp-vs-max}-\eqref{prop::sp-vs-max::1} yields the assertion.
\end{proof}

\noindent
The following result shows how to compute the measure of singletons in~$\supp(\nu_L)$.
For compact $K\subseteq X(A)$ and $\alpha\in K$, we denote by $[\alpha]_K$ the set of all $a\in A$ for which $\hat{a}\!\restriction_K$ attains its maximum at $\alpha$, i.e., $[\alpha]_K:=\{a\in A:\|a\|_K=\left|\alpha(a)\right|\}$. 
Note that $A=\bigcup_{\alpha\in K}[\alpha]_K$.

\begin{thm}
\thlabel{thm::thm}
Let $L\colon A\to\RR$ be represented by the Radon measure $\nu_L$ with compact support $K$ and let $d\in\NN$. 
Then
$$
\nu_L(\{\alpha\})=\max\{\lambda\in [0,1]: \sqrt[2^d]{\lambda} p_L \leq p_d \text{ on }[\alpha]_K\}\qquad\text{ for all }\alpha\in K.
$$
\end{thm}
\begin{proof}
Let $\alpha\in K$ and recall that $p_L=\|\cdot\|_K$ by~\thref{char::compact::support}.

Now, let $a\in [\alpha]_K$ and set $\lambda_\alpha:=\nu_L(\{\alpha\})\in [0,1]$.
Then 
$$
\lambda_\alpha p_L(a)^{2^d} = \nu_L(\{\alpha\}) \left|\alpha(a)\right|^{2^d} \leq \int_K \alpha(a)^{2^d} \dd\nu_L = L(a^{2^d})
$$
yields that $\sqrt[2^d]{\lambda_\alpha} p_L(a)\leq p_d(a)$, i.e., $\lambda_\alpha\in\{\lambda\in [0,1]: \sqrt[2^d]{\lambda} p_L \leq p_d \text{ on }[\alpha]_K\}$.

Conversely, let $\lambda\in [0,1]$ be such that $\sqrt[2^d]{\lambda} p_L \leq p_d$ on $[\alpha]_K$ and let~$0<\varepsilon\leq 1$.
Since $\nu_L$ is outer regular and sets of the form $U(b)=\{\alpha\in X(A): \hat{b}(\alpha)>0\}$ with $b\in A$ are a basis of~$\tau_{X(A)}$ (cf.\ \eqref{eq::basis-weak-top}), there exists~$b\in A$ such that $\alpha\in U(b)$ and $\nu_L(U(b))\leq\nu_L(\{\alpha\})+\varepsilon$. 
Set $U:=U(b)$ and note that $p_L(b)=\|b\|_K\geq\left|\alpha(b)\right|>0$. Let~$n\in\NN$ be such that
$$
\left(1-\frac{\varepsilon\alpha(b)^2}{p_L(2b)^2}\right)^{2n}\leq\sqrt[2^d]{\varepsilon}\qquad\text{and set}\qquad a_\varepsilon:=\left(1-\frac{\varepsilon(\alpha(b)-b)^2}{p_L(2b)^2}\right)^{2n}.
$$
By construction, $p_L(a_\varepsilon)=\|a_\varepsilon\|_K=1$ as $\alpha(a_\varepsilon)=1$ and $\left|\alpha(b)-\beta(b)\right|\leq 2p_L(b)$ for all $\beta\in K$, i.e., $a_\varepsilon\in [\alpha]_K$, as well as $\beta(a_\varepsilon)\leq\sqrt[2^d]{\varepsilon}$ for all $\beta\in K\setminus U$.
Thus,
\begin{alignat*}{6}
p_d(a_\varepsilon)^{2^d} &= \int_{K\setminus U}\hat{a}_\varepsilon^{2^d}\dd\nu_L &&+ \int_{K\cap (U\setminus\{\alpha\})}\hat{a}_\varepsilon^{2^d}\dd\nu_L &&+ \int_{\{\alpha\}}\hat{a}_\varepsilon^{2^d}\dd\nu_L\\
&\leq \varepsilon\cdot\nu_L(K\setminus U) &&+ 1\cdot\nu_L(U\setminus\{\alpha\}) &&+ \nu_L(\{\alpha\})\\
&\leq \varepsilon\cdot 1 &&+ 1\cdot\varepsilon &&+ \nu_L(\{\alpha\}).
\end{alignat*}
Therefore, $\lambda\cdot 1= \lambda p_L(a_\varepsilon)^{2^d}\leq p_d(a_\varepsilon)^{2^d}\leq 2\varepsilon + \nu_L(\{\alpha\})$ and hence, $\nu_L(\{\alpha\})\geq \lambda$ as $0<\varepsilon\leq 1$ was arbitrary. 
\end{proof}

\noindent
Let $L\colon A\to\RR$ be represented by the Radon measure $\nu_L$ with compact support~$K$.
If $\nu_L(\{\alpha\})>0$ for all $\alpha\in K$, then $K$ is countable as $\nu_L(K)=L(1)=1$ while the converse is false in general.

\begin{exmp}
Consider the polynomial algebra $\RR[X]$ and the Radon measure~$\nu$ on~$X(\RR[X])\simeq \RR$ given by $\nu:=\sum_{n=1}^\infty 2^{-n}\delta_{n^{-1}}$, where $\delta_{n^{-1}}$ denotes the Dirac measure on~$\RR$ concentrated on $\{n^{-1}\}$ for each $n\in\NN$.
Clearly, $\nu$ is representing for the linear functional given by $f\mapsto\int f\dd\nu$ and its support
$$
\supp(\nu)=\overline{\{n^{-1}:n\in\NN\}}=\{n^{-1}:n\in\NN\}\cup\{0\}
$$
is countable and compact, but $\nu(\{0\})=0$.
\end{exmp}

\noindent
However, if there exists $\lambda\in(0,1]$ such that $\nu_L(\{\alpha\})\geq\lambda>0$ for all $\alpha\in K$, then~$K$ is finite (with $\left|K\right|\leq\lambda^{-1}$) and the converse is also true as the singletons are closed w.r.t.\ $\tau_{X(A)}$.
Therefore, \thref{thm::thm} yields the following two results.

\begin{cor}
\thlabel{thm::thm::cor1}
Let $L\colon A\to\RR$ be represented by the Radon measure $\nu_L$ with compact support~$K$. If for each $\alpha\in K$ there exists $C_\alpha\in (0,1]$ and $d_\alpha\in\NN$ such that $C_\alpha p_L\leq p_{d_\alpha}$ on $[\alpha]_K$, then $K$ is countable.
\end{cor}

\begin{cor}\thlabel{thm::thm::cor2}
Let $L\colon A\to\RR$ be represented by the Radon measure $\nu_L$ with compact support~$K$ and let $d\in\NN$. 
Then $K$ is finite if and only if there exists $C\in (0,1]$ such that $C p_L\leq p_d$.
\noindent
In this case, $K=\sp(p_d)$ and $\left|K\right|\leq C^{-2^d}$.
\end{cor}
\begin{proof}
Note that $K=\sp(p_L)=\sp(p_d)$ by \thref{char::compact::support} as $Cp_L\leq p_d\leq p_L$.
\end{proof}

\noindent
Often one is interested in constructing a representing Radon measure whose support is contained in the topological dual $V^\prime$ of a lc space $(V, \tau)$ rather than in $X(A)$ (see, e.g., \cite{BeSi71}, \cite[Vol. II, Chapter 5, Sect. 2]{BeKo95}, \cite{InKuRo14}, \cite{Schmu18}). 
\noindent
Suppose that $V$ is a subspace of $A$ and there exists a representing Radon measure~$\nu$ on $X(A)$ for a $\sum A^2$--positive linear functional $L\colon A\to\RR$. 
Then the natural embedding of $V$ in $A$ extends to a homomorphism $\phi\colon S(V)\to A$. 
Since $X(S(V))$ is isomorphic to the algebraic dual $V^\ast$ of $V$, the dual map of $\phi$ actually gives a map $\phi^\prime\colon X(A)\to V^\ast$ (see, e.g., \cite[p.10]{GhInKuMa18}. 
If we endow $V^\ast$ with the weak topology, then $\phi^\prime$ is continuous and so $\nu^\prime:={\phi^\prime}_{\#}\nu$ is a representing Radon measure on $V^\ast$ for $L\circ\phi$.
If the support of $\nu$ is compact in $X(A)$ and $p_L$ is $\tau$--continuous, then \thref{char::compact} ensures that $\supp(\nu)=\sp(p_L)$ and $\phi^\prime(\supp(\nu))$ is compact in $V^\prime$ with $\nu^\prime(\phi^\prime(\supp(\nu))=1$.

The previous observations motivated us to investigate more deeply in the next section the case when $A$ is endowed with some topology compatible with the algebra structure.

\section{Topological aspects}\label{sec::topology}

Let $\tau$ be a locally convex topology on the algebra $A$, i.e., $\tau$ is generated by a family of seminorms on $A$. 
The pair $(A,\tau)$ is called a \emph{locally convex topological algebra (lc TA)} if the multiplication in $A$ is separately continuous and an \emph{lc TA with continuous multiplication} if the multiplication is jointly continuous. The pair $(A,\tau)$ is called a \emph{locally multiplicative convex algebra (lmc TA)} if $\tau$ is generated by a family of submultiplicative seminorms on $A$. In this case the multiplication is automatically jointly continuous and so each lmc TA is an lc TA with continuous multiplication (see, e.g., \cite{Mal86, Tre67} for details).

\subsection{Some natural topologies related to the growth condition}\label{sec::natural}

In the following, we construct and characterize two topologies on $A$ closely related to the growth condition~\eqref{thm::MP<->growth::cond}.\par\medskip

\noindent
Let us fix a $\sum A^2$--positive linear functional $L\colon A\to\RR$.
Consider the to\-po\-logy $\tau_\P$ generated by the family $\P:=\{p_d:d\in\NN\}$ of seminorms on $A$. 
Note that $\tau_{\P}$ is Hausdorff if and only if $p_1$ is a norm (see~\thref{thm::main-thm::rem1}-\eqref{rem::sup-power-generators::0}).

\begin{prop}
\thlabel{prop::can-top1}
The topology $\tau_{\P}$ is the weakest topology on $A$ such that~$(A,\tau_{\P})$ is an lc TA with continuous multiplication and $L$ is $\tau_{\P}$--continuous.
\end{prop}
\begin{proof}
By \eqref{eq::CBS}, the linear functional $L$ is $p_1$--continuous and so $L$ is $\tau_\P$--con\-tinuous. 
Since for each $d\in\NN$ we have that $p_d(a\cdot b)\leq p_{d+1}(a)p_{d+1}(b)$ for all~$a,b\in A$, the multiplication is jointly continuous (cf.\ \cite[p.~420]{Tre67}).

Let $(A,\tau)$ be an lc TA with continuous multiplication and $L$ be $\tau$--continuous. 
Then there exists a $\tau$--continuous seminorm~$q$ on~$A$ such that $|L |\leq q$. 
As the multiplication is jointly continuous, for each $d\in\NN$, there exists a $\tau$--continuous seminorm~$r$ on~$A$ such that $ q(a^{2^d})\leq r(a)^{2^d}$ for all $a\in A$, i.e., $p_d\leq r$. Hence, each seminorm in $\P$ is $\tau$--continuous and so $\tau_\P\subseteq\tau$. 
\end{proof}

\noindent
In case the growth condition~\eqref{thm::MP<->growth::cond} holds, $L$ is also continuous w.r.t.\ the submultiplicative seminorm~$p_L$ on $A$ (see \thref{lem::growth->subm}). 
Consider the topology $\tau_L$ generated by $p_L$ and note that $\tau_L$ is Hausdorff if and only if $p_1$ is a norm as $p_L=\sup_{d\in\NN}p_d$.

\begin{prop}
\thlabel{prop::can-top2}
The topology $\tau_L$ is the weakest topology on $A$ such that~$(A,\tau_L)$ is an lmc TA and $L$ is $\tau_L$--continuous.
\end{prop}
\begin{proof}
Let $(A,\tau)$ be an lmc TA and $L$ be $\tau$--continuous. 
Then there exists a $\tau$--continuous submultiplicative seminorm~$q$ on~$A$ such that $L(a^{2^d})\leq q(a^{2^d})\leq q(a)^{2^d}$ for all $d\in\NN$ and all $a\in A$, i.e., $p_L=\sup_{d\in\NN}p_d\leq q$.
Hence, $p_L$ is $\tau$--continuous and so~$\tau_L\subseteq\tau$.
\end{proof}

\noindent
Since the lmc TA $(A,\tau_L)$ is also an lc TA with continuous multiplication, $\tau_\P\subseteq\tau_L$ by \thref{prop::can-top1}. 
Therefore, by \thref{prop::can-top2}, the converse inclusion $\tau_L\subseteq \tau_\P$ holds if and only if $(A,\tau_\P)$ is an lmc TA. 
Note that by~\thref{thm::thm::cor2} and \cite[Proposition~7.7]{Tre67} the case $\tau_\P=\tau_L$ is equivalent to the existence of a representing Radon measure for~$L$ with \emph{finite} support.

\subsection{Generators}\label{sec::generators}

In the following, we investigate the consequences of assuming the growth condition \eqref{thm::MP<->growth::cond} not on all elements of $A$ but just on a proper subset of~$A$, e.g., the generators of $A$ or the generators of a dense subalgebra of $A$ when~$A$ is a topological algebra.\par\medskip

\noindent
Let us fix a $\sum A^2$--positive linear functional $L\colon A\to\RR$.

\begin{prop}
\thlabel{prop::gen1}
Let $A$ be generated by~$\{a_i:i\in I\}$ such that $\sup_{d\in\NN}p_d(a_i)<\infty$ for all~$i\in I$.
Then~\eqref{thm::MP<->growth::cond} holds.
\end{prop}
\begin{proof}
Since $p_d(\lambda a+bc)\leq \left|\lambda\right|p_d(a)+p_{d+1}(b)p_{d+1}(c)$ for all $\lambda\in\RR$ and all $a,b,c,\in A$ by~\thref{thm::main-thm::rem1}-\eqref{rem::sup-power-generators::0}, we easily get that $\sup_{d\in\NN}p_d(a)<\infty$ for all $a\in A$. This yields the assertion as $\sup_{n\in\NN}\sqrt[2n]{L(a^{2n})}=\sup_{d\in\NN}\sqrt[2^d]{L(a^{2^d})}$ for all $a\in A$.
\end{proof}

\noindent
Combining \thref{prop::gen1,char::compact} yields the following result.

\begin{cor}
\thlabel{cor::gen1}
Let $A$ be generated by~$\{a_i:i\in I\}$ such that $\sup_{d\in\NN}p_d(a_i)<\infty$ for all~$i\in I$.
Then~$\nu_L$ is the unique representing Radon measure for~$L$ with compact support. In particular, 
$$
\supp(\nu_L)\subseteq\{\alpha\in X(A):\left|\alpha(a_i)\right| \leq p_L(a_i)\text{ for all } i \in I \}.
$$
\end{cor}

\begin{exmp}
Consider the polynomial algebra $\RR[\uX]:=\RR[X_1,\ldots,X_m]$ for some $m\in\NN$ and let $L\colon \RR[\uX]\to\RR$ be a $\sum\RR[\uX]^2$--positive linear functional such that $\sup_{d\in\NN}p_d(X_i)<\infty$ for all $i\in\{1,\ldots,m\}$. 
Then~$\nu_L$ is the unique representing Radon measure for~$L$ on~$X(\RR[\uX])\simeq\RR^m$ with compact support by \thref{cor::gen1} and, in particular,
$$
\supp(\nu_L)\subseteq\prod_{i=1}^m[-p_L(X_i),p_L(X_i)]\subseteq\RR^m.
$$
Note that further aspects of the problem of bounding the (not necessarily compact) support of a Radon measure on~$\RR^m$ by a closed box are considered in~\cite{Las11}.
\end{exmp}

\noindent
Now, let us fix an lc TA $(A,\tau)$ such that~$L$ is $\tau$--continuous.  Note that here the multiplication in $A$ is not assumed to be jointly continuous. We investigate what can be said when we assume the growth condition only on a dense subalgebra.

\begin{prop}
\thlabel{prop::dense-sub}
Let $B\subseteq A$ be a $\tau$--dense subalgebra such that $\sup_{d\in\NN}p_d(b)<\infty$ for all $b\in B$ and $\tau_{L\!\restriction_B}\subseteq \tau\!\restriction_B$. Then~\eqref{thm::MP<->growth::cond} holds.
\end{prop}

\begin{proof}
Recall that $\tau_{L\!\restriction_B}$ denotes the topology on~$B$ generated by the submultiplicative seminorm~$p_{L\!\restriction_B}$ that is induced by $L\!\restriction_B\colon B\to\RR$ (cf.\ \thref{lem::growth->subm}). Further, $p_{L\!\restriction_B}$ is $\tau\!\restriction_B$--continuous as $\tau_{L\!\restriction_B}\subseteq \tau\!\restriction_B$. Therefore, by the Hahn--Banach theorem, $p_{L\!\restriction_B}$ extends to a $\tau$--continuous seminorm~$p$ on~$A$.
W.l.o.g.\ we can assume that $L$ is $p$--continuous (otherwise replace $p$ by $\max\{p, |L|\}$). By \thref{lem::subm->growth} it suffices to show that~$p$ is submultiplicative.

Let $a_1,a_2\in A$ and $\varepsilon>0$.
Since the multiplication is separately continuous, there exists a $\tau$--continuous seminorm $q_1$ on $A$ such that $p(a\cdot a_1)\leq q_1(a)$ for all $a\in A$.
Then the density of $B$ in $(A, \tau)$ implies that there exists $b_2\in B$ such that $\max\{p,q_1\}(a_2-b_2)\leq\varepsilon$ and so
$$
\left|p(a_1a_2)-p(a_1b_2)\right|\leq p(a_1(a_2-b_2))\leq q_1(a_2-b_2)\leq\varepsilon.
$$
Similarly, there exists a $\tau$--continuous seminorm $q_2$ on $A$ such that $p(a\cdot b_2)\leq q_2(a)$ for all $a\in A$ as well as $b_1\in B$ such that $\max\{p,q_2\}(a_1-b_1)\leq\varepsilon$. As above, $\left|p(a_1b_2)-p(b_1b_2)\right|\leq\varepsilon$.
Since $p\!\restriction_B=p_{L\!\restriction_B}$ and $p(a_i-b_i)\leq\varepsilon$ for $i\in\{1,2\}$, 
$$
p(a_1a_2)\leq p(b_1b_2) + 2\varepsilon\leq p(b_1)p(b_2) + 2\varepsilon\leq (p(a_1)+\varepsilon)(p(a_2)+\varepsilon)+ 2\varepsilon.
$$
Hence, $p$ is submultiplicative as $\varepsilon>0$ was arbitrary.
\end{proof}

\noindent
Combining \thref{prop::gen1,prop::dense-sub} yields the following generalization of \thref{cor::gen1}.

\begin{cor}
\thlabel{thm-gen}
Let $B\subseteq A$ be a $\tau$--dense subalgebra generated by~$\{b_i:i\in I\}$ such that $\sup_{d \in\NN}p_d(b_i)<\infty$ for all $i\in I$ and $\tau_{L\!\restriction_B}\subseteq \tau\!\restriction_B$.
Then~$\nu_L$ is the unique representing Radon measure for~$L$ with compact support. In particular, 
$$
\supp(\nu_L)\subseteq\{\alpha\in X(A):\left|\alpha(b_i)\right| \leq p_L(b_i)\text{ for all } i \in I \}.
$$
\end{cor}

\section*{Acknowledgments}
We are indebted to the Baden--W\"urttemberg Stiftung for the financial support to this work by the Eliteprogramme for Postdocs. This work was also partially supported by the Ausschuss f\"ur Forschungsfragen~(AFF) and Young Scholar Fund~(YSF) 2018 awarded by the University of Konstanz.

\end{document}